\documentclass[a4paper, 11pt, reqno]{amsart}
\usepackage[english]{babel}
\usepackage[utf8]{inputenc}
\usepackage[T1]{fontenc}
\usepackage{a4wide, amsthm, amsmath, amssymb, amsrefs, enumerate, enumitem, mathtools, lmodern, microtype, mathrsfs, listings, tikz-cd, comment, float}
\usepackage{a4wide}
\usepackage{hyperref}
\usepackage{nicefrac}
\usepackage{tikz}
\usetikzlibrary{shapes,positioning,intersections,quotes}
\theoremstyle{plain}
\newtheorem{thm}{Theorem}[section]
\newtheorem{prop}[thm]{Proposition}
\newtheorem{lemma}[thm]{Lemma}

\newtheorem{conj}[thm]{Conjecture}
\theoremstyle{definition}
\newtheorem{defn}[thm]{Definition}

\theoremstyle{remark}
\newtheorem*{rmk}{Remark}

\makeatletter
\let\@@pmod\pmod
\DeclareRobustCommand{\pmod}{\@ifstar\@pmods\@@pmod}
\def\@pmods#1{\mkern4mu({\operator@font mod}\mkern 6mu#1)}
\makeatother

\binoppenalty=\maxdimen
\relpenalty=\maxdimen

\numberwithin{equation}{section}
\setlist{nosep}
\setlist{noitemsep}

\lstdefinelanguage{Sage}[]{Python}
{morekeywords={False,sage,True},sensitive=true}

\newcommand{\C}{\mathbb{C}}
\renewcommand{\H}{\mathbb{H}}
\newcommand{\Z}{\mathbb{Z}}

\newcommand{\N}{\mathbb{N}}
\newcommand{\R}{\mathbb{R}}

\newcommand{\slz}{{\text {\rm SL}}_2(\mathbb{Z})}

\DeclarePairedDelimiter\floor{\lfloor}{\rfloor}

\title{On the Log-Concavity of the D'Arcais Polynomials for Normalised Functions}

\author[Stumpenhusen]{Johann Stumpenhusen}
\address{Department of Mathematics and Computer Science, Division of Mathematics, University of Cologne,
	Weyertal 86-90, 50931 Cologne, Germany}
\email{jstumpen@math.uni-koeln.de}
\subjclass{Primary 11F20; Secondary 05A20.}
\keywords{Dedekind eta function, Generating functions, Recurrence relations.}

\begin{document}

\begin{abstract}
    A sequence $(a_n)_{n \in \N}$ of non-negative real numbers is called \emph{log-concave} at $n$ if $a_n^2 \geq a_{n+1}a_{n-1}$. This property has been generalised in various ways to families of polynomials. We introduce a new variant and show that certain types of D'Arcais polynomials have the respective properties at certain points. 
\end{abstract}

\maketitle

\section{Introduction}

%COPY FROM "On the Detection of Non-Roots of the D'Arcais Polynomials"

\subsection{D'Arcais polynomials.} Let
\[\eta(\tau) \coloneqq q^{\frac{1}{24}}\prod_{n = 1}^\infty (1 - q^n), \qquad q \coloneqq e^{2\pi i \tau}, \tau \in \H,\]
be the Dedekind eta function appearing in multiple areas of combinatorics and number theory because it is closely related to the reciprocal of the generating function of the partition function. As $\eta(\tau)$ is a modular form of weight $\frac{1}{2}$ for $\slz$, it admits a Fourier expansion. The exact properties of the Fourier coefficients of the powers $\eta(\tau)^{-X}$ have been of great interest, dating back to Euler and Jacobi investigating pentagonal and triangular numbers  (see Koehler \cite{Koehler}, Ono \cite{On03}). In fact, we can rewrite the infinite product as
\begin{equation*}\label{eq:DefD'ArcaisPolynomials}
    \prod_{n = 1}^\infty (1 - q^n)^{-X} = \exp \left(X \sum_{n = 1}^\infty \sigma(n) \frac{q^n}{n}\right) =: \sum_{n = 0}^\infty P_n^\sigma(X)q^n
\end{equation*}
where $\sigma(n) \coloneqq \sum_{d \mid n} d$ is the usual sum-of-divisors-function. The polynomials $P_n^\sigma(X)$ are called \textit{D'Arcais polynomials} (or \textit{Nekrasov}--\textit{Okounkov} polynomials in combinatorics) and the central object of this paper.

%END COPY

The above formula is in not unique to $\eta$, in a sense. That is, let $g: \N \longrightarrow \Z$ with $g(1) = 1$, we instead obtain
\begin{equation}\label{eq:ExponentialFormula}
    \sum_{n = 0}^\infty P_n^g(X)q^n \coloneqq \prod_{n = 1}^\infty (1 - q^n)^{-f_g(n)X} = \exp \left(X \sum_{n = 1}^\infty g(n) \frac{q^n}{n}\right)
\end{equation}
where $f_g(n) \coloneqq \frac{1}{n} \sum_{d \mid n}\mu(d)g\left(\frac{n}{d}\right)$ is a rescaled convolution of $g$ with the usual M\"obius function $\mu$. Notice that, by M\"obius inversion, we have $f_\sigma(n) = 1$ for all $n$. In a recent work, the author and Heim \cite{HeimStum26} showed that quite a few techniques that apply in the case of $\sigma$ remain applicable for those functions $g$. This in turn was inspired by work by \.{Z}mija \cite{Zmija} and Heim--Neuhauser \cite{HeimNeu25}.

\subsection{Previous results on log-concavity} While the $P_n^g(X)$ represent Fourier coefficients, these polynomials also have coefficients themselves and their growth behaviour is the central aspect of this paper. We write
\[P_n^g(X) = \sum_{k = 0}^n p_n^g(k)X^k\]
and call a sequence $(a_n)_{n \in \N}$ of positive real numbers \emph{log-concave} at $n_0 \in \N \setminus \{0\}$ if
\[a_{n_0}^2 \geq a_{n_0 - 1}a_{n_0 + 1}.\]
\begin{conj}[Heim--Neuhauser \cite{HeimNeu20}*{Challenge 3}]\label{conj:HeimNeuhauser}
    For every $n \in \N$, the sequence $(p_n^\sigma(k))_{k \in \N}$ is log-concave at each $k \geq 1$.
\end{conj}

Shortly after this conjecture was posed, Abdesselam \cite{Abdessel23} generalised it to a wider family of arithmetic functions. Let $\ell \in \N$ be a natural number and $g_\ell(n)$ the number of subgroups of index $n$ in $\Z^\ell$. These functions satisfy the relation
\[f_{g_\ell}(n) = g_{\ell - 1}(n)\]
for all $n \in \N$. In particular, $g_2 = \sigma$ as the identity $g_\ell(n) = \sum_{d_1 \dots d_\ell = n}d_1^0d_2^1 \dots d_\ell^{\ell - 1}$ reveals that $g_\ell$ is the convolution of increasing powers of the identity function. Abdesselam's generalised version reads as follows.

\begin{conj}[Abdesselam \cite{Abdessel23}*{Conjecture 1.1}]\label{conj:Abdesselam}
    For every $\ell, n \in \N$, the sequence $(p_n^{g_\ell}(k))_{k \in \N}$ is log-concave at each $k \geq 1$.
\end{conj}

 As Abdesselam, Brunialti, Doan, and Velie \cite{AbPruDoVe} point out that the case where $\ell = 1$ is known since $p_n^{g_1}(k)$ equals the unsigned Stirling number of the first kind $c(n,k)$. In their work, they provide a rather inexplicit formula for $p_n^{g_\ell}(k)$ and use this to prove the conjecture in the case where $\ell = 2$ and $k = n - 1$. The case $\ell = 2$ being the most accessible and yet unsolved one also attracted the attention of Hong and Zhang \cite{HongZhang} who considered the asymptotical case in a family of shifted D'Arcais polynomials.

\iffalse
\begin{thm}[Hong--Zhang \cite{HongZhang}*{Theorem 1.1}]\label{thm:HongZhang}
    Let $Q_n^g(X) \coloneqq P_n^g(X + 1) = \sum_{k = 0}^n q_n^g(k) X^k$. For $n$ sufficiently large, we have
    \begin{itemize}
        \item[(a)] For $k \ll \frac{n^{\frac{1}{6}}}{\log n}$, we have $q_n^{g_2}(k)^2 \geq q_n^{g_2}(k - 1)q_n^{g_2}(k + 1)$.
        \item[(b)] For $k \gg \sqrt{n}\log n$, we have $q_n^{g_2}(k) \geq q_n^{g_2}(k + 1)$.
    \end{itemize}
\end{thm}
\fi

Furthermore, both conjectures have been proven to not hold true in all generality. Starr \cite{Starr} presented a more refined approach to show that, in the case $\ell = 2$, Conjecture \ref{conj:Abdesselam} does not hold for $k = 2$ but does hold for $k \geq 3$, eventually.

\begin{thm}[Starr \cite{Starr}*{Corollary 2.4}]\label{thm:Starr}
    Let $n \in \N$ be arbitrary. It holds that
    \begin{equation*}
        \liminf_{n \to \infty} \frac{p_n^{g_2}(2)^2}{p_n^{g_2}(3)p_n^{g_2}(1)} = 0
    \end{equation*}
    and
    \begin{equation*}
        \liminf_{n \to \infty} \frac{p_n^{g_2}(k)^2}{p_n^{g_2}(k+1)p_n^{g_2}(k-1)} > 1
    \end{equation*}
    for $3 \leq k \leq n - 1$.
\end{thm}

\begin{rmk}
    While Starr's approach disproves the conjectures for $\ell = k = 2$, computing an explicit counterexample involves further explicit bounds on convolutions of divisor sums and was done by Charlton, Heim, and the author \cite{CharltonStum}.
\end{rmk}

\subsection{Horizontal and vertical log-concavity} As the $P_n^g(X)$ form a sequence of polynomials of degree $n$ (see Subsection \ref{subsec:RelationDArcais} for more details), we may arrange them in a triangular order as indicated below.
\[
\begin{array}{ccccccccc}
    P_0^g(X)& = & p_0^g(0) & & & & & &\\
    P_1^g(X)& = & p_1^g(0) & + & p_1^g(1)X & & & &\\
    P_2^g(X)& = & p_2^g(0) & + & p_2^g(1)X & + & p_2^g(2)X^2& &\\
    \vdots &  & \vdots &  & \vdots &  & \vdots & & \ddots 
\end{array}
\]
This representation led Heim and Neuhauser to also consider the sequences $(p_n^g(k))_{n \in \N}$.

\begin{defn}
    A sequence $\left(F_n(X)\right)_{n \in \N} \subset \R[X]$ of polynomials where
    \[F_n(X) = \sum_{k = 0}^{\deg(F_n)}f_n(k)X^k\]
    is called
    \begin{itemize}
        \item[(a)] \emph{horizontally log-concave} at $k \in \N \setminus \{0\}$ for $n \in \N$ if $f_n(k)^2 \geq f_n(k + 1)f_n(k - 1)$;
        \item[(b)] \emph{vertically log-concave} at $n \in \N \setminus \{0\}$ for $k \in \N$ if $f_n(k)^2 \geq f_{n + 1}(k)f_{n - 1}(k)$.
    \end{itemize}
\end{defn}

Most of the aforementioned statements thus concern horizontal log-concavity. The latter of those two terms has seen far less usage in recent research but Heim and Neuhauser proved the vertical log-concavity in quite a few cases \cite{HeimNeu20}.

In this article, we also introduce another variant of log-concavity inspired by the arrangement above. We say that the $F_n(X)$ are \emph{skewly log-concave} at $k \in \N$ for $n \in \N_{\geq k + 1}$ if
\begin{equation*}\label{defeq:SkewLogConcave}
    f_n(n - k)^2 \geq f_{n + 1}(n + 1 - k)f_{n - 1}(n - 1 - k)
\end{equation*}
holds.

\section{Statement of Results}\label{sec:StatementOfResults}

For our results, we will usually refer to the renormalised polynomials $A_n^g(X) \coloneqq n!P_n^g(X)$. Canceling out the denominator, these polynomials are elements in $\Z[X]$ and thus sometimes easier to handle. Furthermore, any kind of log-concavity of $A_n^g(X)$ at a specific point implies the same kind of log-concavity for $P_n^g(X)$ at that same point and in the case of horizontal log-concavity, they are even equivalent.

\subsection{Horizontal and vertical log-concavity} Our first result proves the horizontal and vertical log-concavity at certain points for a large family of arithmetic functions and may be viewed as a reinforcement of previous results by Abdesselam, Brunialti, Doan, and Velie \cite{AbPruDoVe}*{Proposition 3.1} and Starr's Theorem \ref{thm:Starr}. 

\begin{thm}\label{thm:GeneralLogConcavity}
    Let $g$ be a normalised $\Z$-valued arithmetic function such that $g(2) \neq 0$ and $k \in \N$ be a natural number. Then, for each $k$, there exist $n_h(g,k), n_v(g,k) \in \N$ such that $A_n^g(X)$ as well as $P_n^g(X)$ are horizontally log-concave at $n - k$ for each $n \geq n_h(g,k)$ and vertically log-concave at each $n \geq n_v(g,k)$ for $n - k$.
\end{thm}

\begin{rmk} Abdesselam, Brunialti, Doan, and Velie \cite{AbPruDoVe}*{Proposition 3.1} showed that $n_h(g_\ell,1) = 3$ for all $\ell \geq 1$. Due to work by Charlton, Heim, and the author \cite{CharltonStum}, we also know that
    \[n_h(g_2,65\,214\,507\,758\,398) > 65\,214\,507\,758\,400.\]
\end{rmk}

%\subsection{Hong--Zhang} Similarly, we ask whether it is possible to generalize Theorem \ref{thm:HongZhang} to not just $g_\ell$ for all $\ell \in \N$ but any normalized function $g$. There should be a dependence on some values of $g$. (OPTIONAL)

\subsection{Skew log-concavity} Our second result establishes some first properties of the $A_n^g(X)$ regarding skew log-concavity. In most of the research so far, the coefficients were considered starting from the \emph{bottom} rather than the \emph{top}. We justify the usefulness of our new perspective by the remarkable simplicity of the following theorem.

\begin{thm}\label{thm:SkewLogConcavity}
    Let $g$ be a normalised $\N$-valued arithmetic function and $k \in \N$.
    \begin{itemize}
        \item[(a)] If there exists $\kappa \in \N$ such that $g(\kappa) \neq 0$ and $g(m) = 0$ for $2 \leq m \leq \kappa$, then there exists an $n_0 \in \N$ such that $A_n^g(X)$ as well as $P_n^g(X)$ are skewly log-concave at $k$ for $n \geq n_0$.
        \item[(b)] If $g = e$,\footnote{The identity with respect to the Dirichlet convolution.} then $A_n^e(X)$ and $P_n^e(X)$ are skewly log-concave everywhere.
    \end{itemize}
\end{thm}

\section{Preliminaries}

\subsection{Arithmetic functions} A function $g: \N \longrightarrow \C$ is called an \emph{arithmetic function}. We say that $g$ is \emph{normalised} if $g(1) = 1$. In particular, an arithmetic function is a pointwise multiple of a normalised function if and only if it is invertible with respect to the Dirichlet convolution.

We use the following notation for some standard arithmetic functions:

\begin{itemize}
    \item the $\ell$-th sum-of-divisors-function $\sigma_\ell(n) \coloneqq \sum_{d \mid n} d^\ell$ with the special case $\sigma \coloneqq \sigma_1$;
    \item the identity with respect to the Dirichlet convolution $e(n) \coloneqq \begin{cases}1, &n = 1,\\ 0, &\textup{otherwise;}
    \end{cases}$
    \item the $\ell$-th power-function $I_\ell(n) \coloneqq n^\ell$ with the special case $E \coloneqq I_0$.
\end{itemize}

\subsection{Recurrence relation of the D'Arcais polynomials} \label{subsec:RelationDArcais}
Equation \eqref{eq:ExponentialFormula} can be derived invoking the identity
\[\log (1 - q) = - \sum_{n = 1}^\infty \frac{q^n}{n}.\]
Applying the differential operator $\frac{\partial}{\partial q}$ then yields
\[P_n^g(X) = \frac{X}{n} \sum_{m = 1}^n g(m)P_{n-m}^g(X)\]
which transfers down to the respective coefficients by
\begin{equation*}\label{eq:CoefficientRecurrenceP}
    p_n^g(k) = \frac{1}{n} \sum_{m = 1}^{n - k + 1} g(m)p_{n-m}^g(k-1).
\end{equation*}
For $A_n^g(X) = n!P_n^g(X)$, this means that
\begin{equation}\label{eq:CoeffcientRecurrenceA}
    a_n^g(k) = \sum_{m = 1}^{n - k + 1}\frac{(n - 1)!}{(n - m)!}g(m)a_{n-m}^g(k - 1).
\end{equation}

\subsection{Coefficients of the D'Arcais polynomials}

The study of the coefficients of the D'Arcais polynomials set foot onto a higher level with the following theorem by Heim and Neuhauser which offers a rather explicit formula for $A_n^g(X)$.

\begin{thm}[Heim--Neuhauser \cite{HeimNeu25a}*{Theorem 1.2}]\label{thm:CoefficientsExplicitFormula}
    Let $g$ be a normalised arithmetic function. We have
    \[\frac{a_n^g(k)}{n!} = \sum_{\mu = (\mu_1, \ldots, \mu_{l(\mu)}) \in P(n - k)} \prod_{j = 1}^{l(\mu)} g(\mu_j + 1) \prod_{i = 0}^{l(\mu) + n - m - 1}(n - i) \sum_{\lambda \in S_{l(\mu)}\mu} \prod_{j = 1}^{l(\mu)} \left(j + \sum_{i = 1}^j \lambda_i\right)\]
    where $P(n - k)$ is the set of partitions of $n - k$, $l(\mu)$ is the length of a partition and $S_{l(\mu)}$ is the symmetric group.
\end{thm}

Reordering and simplifying this result leads us to the first part of our proof.

\section{Proofs}

This section is dedicated to the proofs of our results.

\subsection{Polynomial representation} The first steps are a simple lemma from discrete analysis and a technical one about a certain recurrence.

\begin{lemma}\label{lem:DifferenceIsAPolynomial}
    Let $d \in \N$ and $f: \N \to \R$ such that there exists a polynomial $p_1(X) \in \R[X]$ of degree $d - 1$ satisfying $p_1(n) = f(n) - f(n - 1)$ for all $n \in \N$. Then there exists a polynomial $p_0(X) \in \R[X]$ of degree $d$ satisfying $p_0(n) = f(n)$ for all $n \in \N$.
\end{lemma}

\begin{proof}
    Rearranging the given equation, we obtain
    \[f(n) = p_1(n) + f(n - 1)\]
    and hence
    \[f(n) = f(0) + \sum_{m = 1}^n p_1(m)\]
    inductively. Using partial summation for the sum yields the claim.
\end{proof}

\begin{lemma}\label{lem:SpecialRecurrenceSolution}
    Let $\kappa \in \N_{\geq 2}$ and the sequence $(a_n)_{n \in \N}$ be defined via the recurrence
    \[
        a_n \coloneqq \begin{cases}0, &0 \leq n \leq \kappa - 2;\\
        1 + \max_{\kappa \leq m \leq n + 1}\{a_{n + 1 - m} + m - 1\}, &n \geq \kappa - 1.
    \end{cases}\]
    Then
    \[a_n = n + \left\lfloor\frac{n}{\kappa - 1}\right\rfloor\]
    for $n \geq \kappa - 1$.
\end{lemma}

\begin{proof}
    This is easily verifiable using the sequence
    \[b_n \coloneqq \begin{cases} 0, &0 \leq n \leq \kappa - 2;\\
        a_n - n = 1 + \max_{\kappa \leq m \leq n + 1}\{b_{n + 1 - m}\}, &n \geq \kappa - 1.
    \end{cases}\]
    which increases by 1 every $(\kappa - 1)$-th step.
\end{proof}

Both lemmata will now be combined to help us prove the most central auxiliary result in this article.

\begin{prop}\label{prop:DegreeOfPolynomials}
    Let $g$ be a normalised $\Z$-valued arithmetic function.
    \begin{itemize}
        \item[(a)] If $g$ only takes non-negative values and there exists $\kappa \in \N_{\geq 2}$ such that $g(\kappa) \neq 0$ and $g(m) = 0$ for $2 \leq m \leq \kappa - 1$, then there exists a polynomial $c_k^g(X) \in \R[X]$ of degree $k + \floor{\frac{k}{\kappa - 1}}$ such that $a_n^g(n - k) = c_k^g(n)$ for all $n$.
        \item[(b)] If no such $\kappa$ exists, then $a_n^g(n - k)$ is constant.
        \item[(c)] If $\kappa = 2$, then the leading coefficient of $c_k^g(X)$ is given by $\frac{g(2)^k}{2^kk!}$.\footnote{This case does not require $g$ to only take non-negative values.}
    \end{itemize}
\end{prop}

\begin{proof}
    All items follow by the same inductive argument. We note that 
    \[a_n^g(n) = 1\]
    for every $n \in \N$ and invoke Equation \eqref{eq:CoeffcientRecurrenceA}. Rearranging yields
    \begin{equation*}\label{eq:RecurrenceIntoDifference}
        a_n^g(n-k) - a_{n-1}^g(n - 1 - k) = \sum_{m = \kappa}^{k + 1} \frac{(n - 1)!}{(n - m)!}g(m)a_{n-m}^g(n - m - (k + 1 - m))
    \end{equation*}
    which is, by virtue of Lemma \ref{lem:SpecialRecurrenceSolution}, the zero polynomial in the case of item (b) or a polynomial of degree $k + \floor{\frac{k}{\kappa - 1}}$ in $n$ otherwise. Note that $\frac{(n - 1)!}{(n - m)!}$ is of degree $m - 1$ in $n$. Thus, the assertions follow by Lemma \ref{lem:DifferenceIsAPolynomial}.
\end{proof}

\begin{rmk}
    In the case of $\kappa \in \N_{\geq 3}$, the formula for the leading coefficient gets more involved. In particular, if $g$ is not assumed to only take non-negative values, cancellations may appear, rendering item (a) invalid.
\end{rmk}

By this last result, we acquire quite satisfactory knowledge about the behaviour of the coefficients $a_n^g(n-k)$ which will be exploited in the subsections to come.

\subsection{Horizontal and vertical log-concavity} First, we will tackle the results concerning the previously investigated concepts of horizontal and vertical log-concavity. Recalling the definition of these, Proposition \ref{prop:DegreeOfPolynomials} encourages us to replace the defining inequalities by
\begin{equation*}\label{eq:DefnInequalitiesAsPolynomials}
    c_k^g(n)^2 \geq c_{k+1}(n)c_{k-1}(n) \qquad \textup{and} \qquad c_k^g(n)^2 \geq c_{k+1}(n + 1)c_{k-1}(n - 1)
\end{equation*}
for horizontal and vertical log-concavity, respectively.

\begin{proof}[Proof of Theorem \ref{thm:GeneralLogConcavity}]
    Invoking item (c) of Proposition \ref{prop:DegreeOfPolynomials}, we obtain
    \[\frac{g(2)^k}{2^kk!}, \qquad \frac{g(2)^{k+1}}{2^{k+1}(k+1)!}, \qquad \frac{g(2)^{k-1}}{2^{k-1}(k-1)!}\]
    as the leading coefficients of $c_k^g(X)$, $c_{k+1}^g(X)$, and $c_{k-1}^g(X)$, respectively. Hence the leading coefficient of
    \[c_k^g(X)^2 - c_{k+1}^g(X)c_{k-1}^g(X)\]
    equals
    \[\frac{g(2)^{2k}}{2^{2k}(k+1)!k!} > 0.\]
    This proves horizontal log-concavity for $A_n^g(X)$.
    
    For vertical log-concavity, the situation reveals itself to be quite similar since substituting $n$ by $n + 1$ or $n - 1$ does not affect the leading coefficient at all.

    As mentioned at the beginning of Section \ref{sec:StatementOfResults}, these imply the same properties for $P_n^g(X)$.
\end{proof}

\begin{rmk}
    Note that while the number $n_h(g,k)$ may be chosen to be optimal for both families of polynomials simultaneously, $P_n^g(X)$ may be vertically log-concave for some $n$ for which $A_n^g(X)$ is not.
\end{rmk}

\subsection{Skew log-concavity} We move on to a rather technical lemma concerning log-concavity of polynomials as functions.

\begin{lemma}\label{lem:CoefficientOff(X)^2 - f(X+1)f(X-1)}
    Let $f(X) = \sum_{k = 0}^d a_k X^k \in \R[X]$ be a polynomial of degree $d$. Then
    \[f(X)^2 - f(X + 1)f(X - 1) = \begin{cases}
        0, &d = 0;\\
        2da_d^2X^{2d - 2} + O(X^{2d - 3}), &d \geq 1.
    \end{cases}\]
\end{lemma}

In a sense, for any polynomial over $\R$, when viewed as a function, the region in which it is not log-concave is bounded.

\begin{proof}[Proof of Theorem \ref{thm:SkewLogConcavity}]
    Regarding item (a), we invoke item (a) of Proposition \ref{prop:DegreeOfPolynomials} to obtain that the coefficients $a_n^g(n - k)$ are in fact polynomials in $n$ which will become log-concave from a certain $n_0$ onward due to Lemma \ref{lem:CoefficientOff(X)^2 - f(X+1)f(X-1)}.

    Item (b) follows directly from item (b) of Proposition \ref{prop:DegreeOfPolynomials} since $e(m) = 0$ for all $m \geq 2$.
\end{proof}

\section{Exemplary computations and outlook}

Lastly, we will have a look at a few examples and suggest some further investigatory objectives.

\subsection{Some examples} In this subsection, we will pay a closer look to explicit examples.

\begin{center}
\begin{tabular}{c||c|c|c}
     $g$ & $a_n^g(n)$ & $a_n^g(n - 1)$ & $a_n^g(n - 2)$ \\
     \hline 
     general $g$ & 1 & $\frac{g(2)(n-1)n}{2}$ & $\frac{(n - 2)(n - 1)n(3g(2)^2(n - 3) + 8g(3))}{24}$\\  
     $e$& 1 & 0& 0\\
     $E$& 1 & $\frac{(n-1)n}{2}$& $\frac{(n - 2)(n - 1)n(3n - 1)}{24}$ \\
     $I_\ell$& 1 & $2^{\ell-1}(n-1)n$& $(n - 2)(n - 1)n(2^{2\ell - 3}(n - 3) +  3^{\ell -1})$\\
     $\sigma_\ell$& 1 & $\frac{(2^{\ell} + 1)(n-1)n}{2}$ & $\frac{(n - 2)(n - 1)n(3(2^{\ell} + 1)^2(n - 3) + 8(3^\ell + 1))}{24}$
\end{tabular}
\end{center}

While the case where $g = e$ is extendable to be $a_n(n - k) = 0$ for any $k \in \N$ by item (b) of Proposition \ref{prop:DegreeOfPolynomials}, explicit computations are needed to determine the polynomials in other cases due to recurrence relations.

\subsection{Interdependence} Suppose $F_n(x) \coloneqq \sum_{k = 0}^{\deg(F_n)} f_n(K)X^k$ is a family of polynomials that are vertically log-concave at $n \in \N \setminus \{0\}$ for $k \in \N \setminus \{0\}$ and horizontally log-concave at $k$ for $n \pm 1$. Consider the inequality
\[f_n(k)^4 \geq f_{n-1}(k)^2f_{n+1}(k)^2 \geq f_{n-1}(k-1)f_{n-1}(k + 1)f_{n + 1}(k - 1)f_{n+1}(k+1).\]
If $f_n(k)^2 \leq f_{n - 1}(k + 1)f_{n+1}(k - 1)$ (i.e. it is some sort of ``anti-skewly log-convex''), then $(F_n(X))_{n \in \N}$ is skewly log-concave at $k \in \N \setminus \{0\}$ for $n \in \N \setminus \{0\}$. While this additional property of convexity certainly is sufficient, one may ask which kind of equivalent condition exists.

\end{document}